\documentclass[a4paper, 12pt]{amsart}

\usepackage{amssymb}
\usepackage{amsmath}
\usepackage{amscd}
\usepackage{amsthm}
\usepackage[centertags]{amsmath}
\usepackage{amsfonts}
\usepackage{newlfont}
\usepackage[all]{xy}
\usepackage{graphicx}
\usepackage{amsfonts, amssymb}
\usepackage[usenames]{color}
\usepackage{mathrsfs}
\usepackage{latexsym}

\usepackage{mathabx}

\newtheorem{thm}{Theorem}[section]
\newtheorem{cor}[thm]{Corollary}
\newtheorem{lem}[thm]{Lemma}
\newtheorem{prop}[thm]{Proposition}

\theoremstyle{definition}

\newtheorem{rem}[thm]{Remark}

\newcommand{\zC}{\mathbb C}

\newcommand{\zR}{\mathbb R}
\newcommand{\zN}{\mathbb N}
\newcommand{\zK}{\mathbb K}

\newcommand{\HH}{\mathcal H}

\begin{document}
\baselineskip=.75cm

\title[Symmetric multilinear forms on Hilbert spaces]{Symmetric multilinear forms on Hilbert spaces: where do they attain their norm?}

\thanks{This project was supported in part by CONICET PIP 11220130100329CO, ANPCyT PICT 2015-2299 and UBACyT 20020130100474. }

\author[D. Carando]{Daniel Carando}
\address{Departamento de Matem\'{a}tica - Pab I,
Facultad de Cs. Exactas y Natu\-ra\-les, Universidad de Buenos Aires,
(1428) Buenos Aires, Argentina and IMAS-CONICET}
\email{dcarando@dm.uba.ar}

\author[J. T. Rodr\'{i}guez]{Jorge Tom\'as Rodr\'{i}guez}
\address{Departamento de Matem\'{a}tica and NUCOMPA, Facultad de Cs. Exactas, Universidad Nacional del Centro de la Provincia de Buenos Aires, (7000) Tandil, Argentina and CONICET}
\email{jtrodrig@dm.uba.ar}

\begin{abstract} We characterize the sets of norm one vectors $\mathbf{x}_1,\ldots,\mathbf{x}_k$ in a Hilbert space $\mathcal H$ such that there exists a $k$-linear symmetric form attaining its norm at $(\textbf{x}_1,\ldots,\mathbf{x}_k)$. We prove that in the bilinear case, any two vectors satisfy this property. However, for $k\ge 3$ only collinear vectors satisfy this property in the complex case, while in the real case this is equivalent to $\mathbf{x}_1,\ldots,\mathbf{x}_k$ spanning a subspace of dimension at most 2. We use these results to obtain some applications to symmetric multilinear forms, symmetric tensor products and the exposed points of the unit ball of $\mathcal L_s(^k\mathcal{H})$.
\end{abstract}

\maketitle

\section*{Introduction}
In the recent years, there has been an increasing interest in the study of norm-attaining linear and multilinear operators, as well as other nonlinear mappings. In these different settings, versions of (and counterexamples to) the classical results of Bishop, Phelps, Bollob\'as and Lindenstrauss have appeared, taking into account the particularities of the different classes of mappings under consideration (see for example \cite{acosta2006denseness, acosta2008bishop, aron2015bishop, dantas2017bishop} and the references therein). In this work, we approach the question of norm-attainment from a different point of view: we characterize the $k$-tuples of vectors in a Hilbert space where a symmetric multilinear form attains its norm at it. In order to be more precise, let us introduce some definitions.

In what follows, $\mathcal H$ denotes a Hilbert space over $\zK$, where $\zK$ stands for either the field of complex numbers $\zC$ or the field of real numbers $\zR$. Given a continuous $k$-linear form $T: \underbrace{\mathcal H\times\cdots\times \mathcal H}_k \rightarrow  \zK$, its norm is given by
$$\Vert T\Vert := \sup \{|T(\mathbf{w}_1,\ldots,\mathbf{w}_k)|:\Vert \mathbf{w}_1\Vert, \ldots,\Vert \mathbf{w}_k\Vert \leq 1\}.$$
With this norm, the space $\mathcal L _s(^k\mathcal H)$  of all continuous \emph{symmetric} $k$-linear forms on $\mathcal H$ is a Banach space.
We say that a nonzero $T\in \mathcal L _s(^k\mathcal H)$ \emph{attains its norm} if there are norm one vectors $\mathbf{x}_1,\ldots,\mathbf{x}_k\in \mathcal H$ such that
\begin{equation}\label{def-nat}\Vert T\Vert=|T(\mathbf{x}_1,\ldots,\mathbf{x}_k)|.\end{equation}
In this case we say that $T$ attains its norm at $(\mathbf{x}_1,\ldots,\mathbf{x}_k)$.

A classical result on Hilbert spaces  asserts that the norm of a symmetric $k$-linear form coincides with that of its associated $k$-homogeneous polynomial. This was already shown by Banach in \cite{banach1938homogene} (see also \cite{bochnak1971polynomials}, \cite[Proposition 1.44]{LibroDi99} and, for a more constructive proof, \cite{pappas2007norm}). In other words, if $T\in  \mathcal L _s(^k\mathcal H)$, then
\begin{equation}\label{eq-isometria}\Vert T\Vert =\sup \{|T(\mathbf{w},\ldots,\mathbf{w})|:\Vert \mathbf{w}\Vert \leq 1\}.\end{equation}
From \eqref{eq-isometria} it is not hard to see that a norm-attaining symmetric $k$-linear form attains its norm at some $k$-tuple of the form $(\mathbf{w},\ldots,\mathbf{w})$. To see this, it is enough to reduce the problem to a finite dimensional setting (taking the span of the vectors satisfying \eqref{def-nat}). But in this case, a natural question arises: is  it possible for such a $k$-linear form to attain its norm at some other  vectors $ \mathbf{x}_1,\ldots,\mathbf{x}_k$ which are not collinear?
For example, the bilinear form $$T(\mathbf{x},\mathbf{y})=x_1 y_1-x_2y_2$$ attains its norm at any pair of norm one vectors of the form $(a,b), (a,-b)$, which can obviously be chosen not to be collinear. As we see in Lemma~\ref{lema bilineales}, any bilinear form on $\mathbb K^2$ attaining its norm in non-collinear vectors looks like $T$ in some appropriate basis.

Therefore, the aim of this article is to characterize the vectors $ \mathbf{x}_1,\ldots,\mathbf{x}_k$ for which there exists some symmetric $k$-linear form attains its norm at $ (\mathbf{x}_1,\ldots,\mathbf{x}_k)$.
We show in Proposition~\ref{bilineales} that if $k=2$, any two vectors will work. However, for $k\ge 3$ the situation is different, as stated in Theorem~\ref{teorema_principal}: in the complex case, the vectors $ \mathbf{x}_1,\ldots,\mathbf{x}_k$ must be collinear, while in the real case they must span a subspace of dimension at most 2. These conditions are necessary and sufficient.

The article is organized as follows. We state our main result (Theorem~\ref{teorema_principal}) in Section~1. We also  present and prove some consequences: a Bollob\'as-like result on where multilinear forms \emph{almost} attain their norm and some applications to complexification, symmetric tensor products and geometry of spaces of multilinear forms.
The proof of Theorem~\ref{teorema_principal} will be given along Sections 3 and 4. In Section 2 we deal with bilinear forms, while $k$-linear forms will be treated in Section 3 (in the complex case) and Section  4 (in the real case).



\section{Main results and applications}\label{seccion principal}

The following is our main result.
\begin{thm}\label{teorema_principal} Let $k\ge 3$ and $\mathbf{x}_1,\ldots,\mathbf{x}_k$ be norm one vectors on a  Hilbert space $\mathcal H$ over $\zK$.
There exists a symmetric $k$-linear form $T$ on $\mathcal H$  attaining its norm at $(\mathbf{x}_1,\ldots, \mathbf{x}_k)$ if and only if $$\dim (\operatorname{span}\{\mathbf{x}_1,\ldots,\mathbf{x}_k\})= \begin{cases} 1 \quad\quad\quad \text{ if }\zK=\zC \\ 1 \text{ or }2 \quad\text{ if }\zK=\zR.\end{cases}
$$
\end{thm}

It is not hard to see that Theorem~\ref{teorema_principal} also holds for vector valued multilinear forms. Indeed, suppose that $E$ is a Banach space and $T\in \mathcal L _s(^k\mathcal H,E)$ attains its norm at $(\mathbf{x}_1,\ldots, \mathbf{x}_k)$. If we take a norm one linear function $\varphi\in E^*$, with $\varphi (T(\mathbf{x}_1,\ldots, \mathbf{x}_k))=\Vert T(\mathbf{x}_1,\ldots, \mathbf{x}_k)\Vert$, then $\varphi\circ T \in \mathcal L _s(^k\mathcal H)$ also attains its norm at $(\mathbf{x}_1,\ldots, \mathbf{x}_k)$, so the \emph{only if} part follows. For the \emph{if} part, just multiply the multilinear form whose existence is guaranteed by the theorem by any nonzero vector in $E$.

As we will see in Proposition~\ref{bilineales}, for any pair of vectors in $\mathcal H$ there exists a bilinear form attaining its norm in that pair. It is interesting to note that  this characterizes real Hilbert spaces (see Proposition  2.9 of \cite{benitez1993characterization}).

As we mentioned above, Sections \ref{seccion compleja} and \ref{seccion real} will be devoted to the proof of the Theorem~\ref{teorema_principal} in the complex and real case respectively. While the \emph{if} part is trivial in the complex case (where all the vectors are collinear), it is rather involved in the real case, when they span a 2-dimensional subspace.

Now we prove some consequences of our main theorem. The first one is a quantitative approximate version of Theorem~\ref{teorema_principal}, much in the Bollob\'as spirit: if $T$ almost attain its norm at $\mathbf{x}_1,\ldots,\mathbf{x}_k$, then the vectors $\mathbf{x}_1,\ldots,\mathbf{x}_k$ must be close to a 2-dimensional subspace (or, in the complex case, to a 1-dimensional subspace).

\begin{cor} For each $k\ge 3$ and  $\varepsilon >0$ there exist $\beta=\beta(k,\varepsilon)>0$  
such that the following holds:
 if $T$ is a norm one symmetric $k$-linear form on $\mathcal H$ and $\mathbf{x}_1,\ldots,\mathbf{x}_k$ are norm one vectors in $\mathcal H$  with $$\vert T(\mathbf{x}_1,\ldots,\mathbf{x}_k) \vert > 1 - \beta,$$ then there exist a subspace $\mathcal V\subset \operatorname{span}\{\mathbf{x}_1,\ldots,\mathbf{x}_k\}$ of dimension at most 2 (which is 1-dimensional in the complex case) such that the distance of each $\mathbf{x}_j$ to $\mathcal V$ is at most $\varepsilon$.
\end{cor}
\begin{proof}
Finite dimensional Banach spaces enjoy the Bishop-Phelps-Bollob\'as property for symmetric $k$-linear forms \cite[Proposition 2.3]{dantas2017bishop}. Then, given $k$ and $\varepsilon$ as in the statement, let $\beta=\eta(\varepsilon)$, where $\eta$ corresponds to the Bishop-Phelps-Bollob\'as property for symmetric $k$-linear forms on a Hilbert space of dimension (at most) $k$. Set $\mathcal{H}_1:=\operatorname{span}\{\mathbf{x}_1,\ldots,\mathbf{x}_k\}$ 
and consider $T_1:=\frac{T|_{\mathcal{H}_1}}{\|T|_{\mathcal{H}_1}\|}$. Then, $T_1$ is a norm one $k$-linear form on a finite dimensional Hilbert space satisfying $$\vert T_1(\mathbf{x}_1,\ldots,\mathbf{x}_k) \vert > 1 - \beta=1-\eta(\varepsilon).$$

The Bishop-Phelps-Bollob\'as property for symmetric $k$-linear forms and the definition of $\eta$ gives a norm one symmetric $k$-linear form $S$ and norm one vectors $\mathbf{z}_1,\ldots,\mathbf{z}_k$ in $\mathcal{H}_1$ such that $$|S(\mathbf{z}_1,\ldots,\mathbf{z}_k)|=1\quad\text{and}\quad \|\mathbf{z}_j- \mathbf{x}_j\|<\varepsilon\quad\text{for }j=1,\dots,k.$$
Since $S$ attains its norm at $\mathbf{z}_1,\ldots,\mathbf{z}_k$, Theorem~\ref{teorema_principal} implies that $\mathcal{V}=\operatorname{span}\{ \mathbf{z}_1,\ldots,\mathbf{z}_k \} $ has dimension at most 2 in the real case (dimension 1 in the complex one) and the result follows.
\end{proof}

Another consequence of Theorem~\ref{teorema_principal} is related to the symmetric projective tensor norm in Hilbert spaces. For an introduction to this topic and the notation we use next,  we refer the reader to Floret's survey \cite{floret1997natural}.  Since the space of  $k$-homogeneous polynomials $ \mathcal P (^k\mathcal H)$ is isometric to the space of symmetric $k$-linear forms $\mathcal L _s(^k\mathcal H)$ (see \eqref{eq-isometria}), the symmetric projective tensor norm of $\mathbf{x}_1\vee\cdots \vee\mathbf{x}_k$ can be computed as follows:
\begin{eqnarray*}
\pi_s(\mathbf{x}_1\vee\cdots \vee\mathbf{x}_k, \otimes^{k,s} \mathcal H) &=&\max \{|P(\mathbf{x}_1\vee\cdots \vee\mathbf{x}_k)|: P\in \mathcal P (^k\mathcal H),\  \Vert P \Vert =1\} \\
&=& \max \{|T(\mathbf{x}_1,\ldots,\mathbf{x}_k)|: T\in \mathcal L _s(^k\mathcal H),\  \Vert T \Vert =1\}.
\end{eqnarray*}
This means that $\mathbf{x}_1\vee\cdots \vee\mathbf{x}_k$ is a norm one element of $\otimes_{\pi_s}^{k,s} \mathcal H$ if and only if there exists a symmetric $k$-linear form attaining its norm at $(\mathbf{x}_1,\ldots,\mathbf{x}_k)$. Now we can apply Theorem ~\ref{teorema_principal} to obtain the following.

\begin{prop}\label{prop tensores} Let $k\ge 3$ and  let $\mathbf{x}_1,\ldots,\mathbf{x}_k $ be norm one vectors in $\mathcal H$. Then,
$\pi_s(\mathbf{x}_1\vee\cdots \vee\mathbf{x}_k, \otimes^{k,s} \mathcal H)=1$  if and only if $$\dim (\operatorname{span}\{\mathbf{x}_1,\ldots,\mathbf{x}_k\})= \begin{cases} 1 \quad\quad\text{ in the complex case} \\  1 \text{ or }2 \quad\text{ in the real case}\end{cases}.
$$
\end{prop}

The previous proposition has a straightforward implication in complexification (see \cite{munoz1999complexifications}) and symmetric tensor products.

\begin{cor}  Let $\mathcal{H}$ be a real Hilbert space and  $\widetilde{\mathcal{H}}$ be the complex Hilbert space obtained from $\mathcal{H}$ by some complexification procedure.  If $k\ge 3$ and $\mathbf{x}_1,\ldots,\mathbf{x}_k\in \mathcal{H}$ are nonzero vectors satisfying
$$\dim (\operatorname{span}\{\mathbf{x}_1,\ldots,\mathbf{x}_k\}) = 2,$$
then $$\pi_s(\mathbf{x}_1\vee\cdots \vee\mathbf{x}_k, \otimes^{k,s} \mathcal H)  > \pi_s(\mathbf{x}_1\vee\cdots \vee\mathbf{x}_k, \otimes^{k,s} \widetilde{\mathcal H}).$$
\end{cor}
\begin{proof}
It is clear that is enough to show the result for norm one vectors. But in this case, the real and complex cases of the previous proposition give that $\pi_s(\mathbf{x}_1\vee\cdots \vee\mathbf{x}_k, \otimes^{k,s} \mathcal H)=1$ while $\pi_s(\mathbf{x}_1\vee\cdots \vee\mathbf{x}_k, \otimes^{k,s} \widetilde{\mathcal H})$ cannot be one (and must be, therefore, less than one).
\end{proof}

In \cite[Lemma 6.2]{friedland2014nuclear} the authors compute the symmetric projective norm  for the tensor $\mathbf{e}_1\vee \mathbf{e}_1\vee \mathbf{e}_2$ both in ${\mathcal H}$ and $\widetilde{\mathcal H}$, showing that they are different. In \cite[Example 6.5]{nie2017symmetric},  this computation is done for $\mathbf{e}_1\vee \mathbf{e}_1 \vee \mathbf{e}_2 \vee \mathbf{e}_2$.
The previous Corollary provides an easy procedure to get examples for which $\pi_s(\mathbf{x}_1\vee\cdots \vee\mathbf{x}_k, \otimes^{k,s} \mathcal H)$ and  $\pi_s(\mathbf{x}_1\vee\cdots \vee\mathbf{x}_k, \otimes^{k,s} \widetilde{\mathcal H})$ do not coincide.

Note that, in the conditions of the previous corollary, if $T\in \mathcal L _s(^k\mathcal H)$ attains its norm at $(\mathbf{x}_1,\ldots,\mathbf{x}_k)$ and  $\widetilde{T}$ is its complexification, then $\Vert \widetilde{T}\Vert > \Vert T\Vert$.

\medskip

The last one is not an application of Theorem \ref{teorema_principal}  but of the auxiliary result Lemma  \ref{lema bilineales}. It concerns the exposed points of $\mathcal L _s(^k\mathcal H)$ for  real two-dimensional Hilbert spaces. For more results regarding the geometry of the unit ball of $\mathcal L _s(^k\mathcal H)$ we refer the reader to the articles \cite{grecu2004geometry, grecu2009unit, kim2003exposed} and the references therein.

\begin{prop}\label{prop-exposed} Let $\mathcal H$ be a  two-dimensional real Hilbert space, $k\geq 2$ and $T\in  \mathcal L _s(^k\mathcal H)$ of norm one attaining its norm at $(\mathbf{x}_1,\ldots,\mathbf{x}_k)$ with $$\operatorname{span}\{\mathbf{x}_1,\ldots,\mathbf{x}_k\} = \mathcal{H}.$$
Then $T$ is a exposed point of the unit ball $\mathcal L _s(^k\mathcal H)$.
\end{prop}

This proposition will be a direct consequence of the Lemma below, which may be of interest in its own right. Before going into the proof, we see that the converse of this Proposition does not hold. Indeed, note that $\frac{1}{2} \otimes^2 \mathbf{e}_1+\frac{1}{2} \otimes^2 \mathbf{e}_2$ exposes the  bilinear  form $L$ given by the inner product of $\zR^2$, and $L$ only attains its norm at $(\mathbf{x},\mathbf{y})$ if $\mathbf{x} =\pm \mathbf{y}$. Similarly $\frac{1}{2} \otimes^3 \mathbf{e}_1+\frac{1}{2} \otimes^3 \mathbf{e}_2$ exposes the trilinear form
$$T(\mathbf{x},\mathbf{y},\mathbf{z})=x_1y_1z_1+x_2y_2z_2,$$
and $T$ only attains its norm at $\mathbf{x},\mathbf{y},\mathbf{z}$ if all the vectors are either $\pm \mathbf{e}_1$ or $\pm \mathbf{e}_2$.

\begin{lem} Let $\mathbf{x}_1,\ldots,\mathbf{x}_k$, with $k\geq 2$, be norm one vectors on a real Hilbert space $\mathcal H$ such that
$$\dim (\operatorname{span}\{\mathbf{x}_1,\ldots,\mathbf{x}_k\}) = 2.$$
If $S, T\in  \mathcal L _s(^k\mathcal H)$ have norm one with
$$S(\mathbf{x}_1,\ldots, \mathbf{x}_k)=T(\mathbf{x}_1,\ldots, \mathbf{x}_k)=1$$
and we set $\mathcal H_1:= \operatorname{span}\{\mathbf{x}_1,\ldots,\mathbf{x}_k\})$, then $S|_{\mathcal{H}_1}=T|_{\mathcal{H}_1}$.
\end{lem}
\begin{proof}
Let us prove this result by induction on $k$. The case $k=2$ follows from Lemma \ref{lema bilineales} below, using that $F$ is unique and determined by $\mathbf{x}_1, \mathbf{x}_2$ and $S(\mathbf{x}_1,\mathbf{x}_2)$. Assume the result holds for $(k-1)$-linear forms. Suppose that
$$S(\mathbf{x}_1,\ldots, \mathbf{x}_k)=T(\mathbf{x}_1,\ldots, \mathbf{x}_k)=1,$$
with $\dim (\operatorname{span}\{\mathbf{x}_1,\ldots,\mathbf{x}_k\}) = 2.$ We may assume that $\mathcal H_1= \operatorname{span}\{\mathbf{x}_1,\mathbf{x}_2\}$. By Lemma \ref{lema bilineales}, applied to the bilinear form $S(\,\,\cdot\,\,, \,\,\cdot\,\,,\mathbf{x}_3,\ldots,\mathbf{x}_k)$,  there is an orthonormal basis $F=\{\mathbf{f}_1,\mathbf{f}_2\}$ of  $\mathcal H_1$ such that
$$S(\mathbf{f}_1, \mathbf{f}_1, \mathbf{x}_3,\ldots,\mathbf{x}_k)=1$$
$$S(\mathbf{f}_2,\mathbf{f}_2, \mathbf{x}_3,\ldots,\mathbf{x}_k)=-1$$
$$S(\mathbf{f}_1,\mathbf{f}_2,\mathbf{x}_3,\ldots,\mathbf{x}_k)=0.$$
Then, given nonzero real numbers $\alpha, \beta$, with $\alpha^2+\beta^2=1$, it is not hard to see that the equations above imply
$$S(\alpha \mathbf{f}_1+\beta \mathbf{f}_2, \alpha \mathbf{f}_1-\beta \mathbf{f}_2,\mathbf{x}_3,\ldots,\mathbf{x}_k)=1.$$
By inductive hypothesis $S(\,\,\cdot\,\,,\ldots, \,\,\cdot\,\,,\mathbf{x}_k)|_{\mathcal{H}_1}=T(\,\,\cdot\,\,,\ldots, \,\,\cdot\,\,,\mathbf{x}_k)|_{\mathcal{H}_1}$, therefore
$$S(\alpha \mathbf{f}_1+\beta \mathbf{f}_2, \alpha \mathbf{f}_1-\beta \mathbf{f}_2,\mathbf{x}_3,\ldots,\mathbf{x}_k)=T(\alpha \mathbf{f}_1+\beta \mathbf{f}_2, \alpha \mathbf{f}_1-\beta \mathbf{f}_2,\mathbf{x}_3,\ldots,\mathbf{x}_k).$$

If we take nonzero $\alpha, \beta\in \zR$ such that both $\alpha \mathbf{f}_1+\beta \mathbf{f}_2$ and $\alpha \mathbf{f}_1-\beta \mathbf{f}_2$ are not $\pm \mathbf{x}_3$, by inductive  hypothesis, we have that
$$S(\,\,\cdot\,\,,\ldots, \,\,\cdot\,\,,\alpha \mathbf{f}_1+\beta \mathbf{f}_2 )|_{\mathcal{H}_1}= T(\,\,\cdot\,\,,\ldots, \,\,\cdot\,\,,\alpha \mathbf{f}_1+\beta \mathbf{f}_2 )|_{\mathcal{H}_1}$$
$$S(\,\,\cdot\,\,,\ldots, \,\,\cdot\,\,,\alpha \mathbf{f}_1-\beta \mathbf{f}_2 )|_{\mathcal{H}_1}=T(\,\,\cdot\,\,,\ldots, \,\,\cdot\,\,,\alpha \mathbf{f}_1-\beta \mathbf{f}_2 )|_{\mathcal{H}_1}.$$
Since $\{\alpha\mathbf{f}_1+\beta \mathbf{f}_2, \alpha\mathbf{f}_1-\beta \mathbf{f}_2 \}$ is a basis of $\mathcal H_1$ we conclude the desired result.
\end{proof}

\emph{Proof of Proposition \ref{prop-exposed}}. \
By the previous Lemma, if we take $\varepsilon:= \operatorname{sign}( T(\mathbf{x}_1,\ldots,\mathbf{x}_k))$, then  $\varepsilon\,\mathbf{x}_1\vee \cdots \vee \mathbf{x}_k \in \otimes^{s,k} \mathcal H$ exposes $T$.
\qed

\section{Bilinear Forms}\label{seccion bilineal}

Given two vectors $\mathbf{x},\mathbf{y}$ in a Hilbert space $\mathcal H$ we write $x\parallel y$ if there is $\lambda \in \zK$ such that $\mathbf{x} = \lambda \mathbf{y}$ and  $x\nparallel y$ when this is not the case. If $\dim(\mathcal{H})=n$, given a symmetric bilinear form $T: \mathcal{H} \times \mathcal{H} \rightarrow \zK$, $\mathbf{x}\in \mathcal H$ and $F=\{\mathbf{f}_1,\ldots,\mathbf{f}_n\}$ a basis of $\mathcal{H}$, $[\mathbf{x}]_F$ stands for the (row) coordinate vector of $\mathbf{x}$ relative to $F$ and  $[T]_F$ the $n\times n$ symmetric matrix such that
$$T(\mathbf{w}_1,\mathbf{w}_2) = [\mathbf{w}_1]_F  [T]_F [\mathbf{w}_2]_F^t.$$
That is, $[T]_F = (T(\mathbf{f}_i,\mathbf{f}_j))_{i,j}$.

\begin{lem}\label{lema bilineales}

Let  $T$ be a norm one symmetric bilinear form on the 2-dimensional Hilbert space $\zK^2$ that attains its norm at $(\mathbf{x},\mathbf{y})$, with $\mathbf{x}\nparallel \mathbf{y}$. Then there is an orthonormal basis $F=\{\mathbf{f}_1,\mathbf{f}_2\}$ of $\zK^2$, such that
$$[T]_F=
\left(\begin{array}{lr}
1 	& 	 0	 \\
0	&	-1		\\
\end{array}\right).$$
As a consequence,  the matrix $[T]_G$ is unitary for every orthonormal basis $G$ of $\zK^2$. Moreover, in the real case, up to signs, the basis $\{\mathbf{f}_1,\mathbf{f}_2\}$ is unique.
\end{lem}

\begin{proof}
Let us prove the existence of $F$ in the complex case (the real case is analogous). Take $\lambda$ a modulus one number such  $\operatorname{Im}(\langle \lambda \mathbf{y}, \mathbf{x} \rangle)=0$. Define $\mathbf{g}_1 := \frac{\mathbf{x}+\lambda \mathbf{y}}{\Vert \mathbf{x}+ \lambda \mathbf{y} \Vert}$ and $\mathbf{g}_2:= \frac{\mathbf{x}-\lambda \mathbf{y}}{\Vert \mathbf{x} - \lambda \mathbf{y} \Vert}$. These vectors are orthonormal:
\begin{eqnarray}
\langle \mathbf{x} +\lambda \mathbf{y},\mathbf{ x} -\lambda \mathbf{y} \rangle & = & \Vert \mathbf{x} \Vert^2 - \Vert \lambda \mathbf{y} \Vert^2 +2 \operatorname{Im}\langle \lambda \mathbf{y}, \mathbf{x} \rangle \nonumber \\
&=& 0 \nonumber.
\end{eqnarray}

Since $T$ attains its norm at $(\mathbf{x},\mathbf{y})$, we have
\begin{eqnarray}
1&=& |T(\mathbf{x},\lambda \mathbf{y})| \nonumber \\
&=& \frac{1}{4}|T(\mathbf{x}+\lambda\mathbf{y},\mathbf{x}+\lambda\mathbf{y}) - T(\mathbf{x}-\lambda\mathbf{y},\mathbf{x}-\lambda\mathbf{y})|\nonumber \\
&\leq &\frac{1}{4} (\Vert \mathbf{x}+\lambda\mathbf{y} \Vert^2 + \Vert \mathbf{x}-\lambda\mathbf{y} \Vert^2)\label{equal} \\
&=& \frac{1}{4} (2\Vert \mathbf{x}\Vert^2 +2 \Vert \lambda\mathbf{y} \Vert^2) =1.  \nonumber \
\end{eqnarray}
Thus \eqref{equal} must be an equality, which implies that
$$\vert T(\mathbf{g}_1,\mathbf{g}_1)\vert =\vert T(\mathbf{g}_2,\mathbf{g}_2)\vert =1.$$

Take $\mathbf{f}_1:=\lambda_1\mathbf{g}_1$ and $\mathbf{f}_2:=\lambda_2\mathbf{g}_2$, where $\lambda_1, \lambda_2$ are modulus one complex numbers such that $\lambda_1^2=\frac{1}{T(\mathbf{g}_1,\mathbf{g}_1)}$ and  $\lambda_2^2=-\frac{1}{T(\mathbf{g}_2,\mathbf{g}_2)}$. Clearly $F=\{\mathbf{f}_1,\mathbf{f}_2\}$ is an orthonormal basis,    and we have $T(\mathbf{f}_1,\mathbf{f}_1)= 1$ and $T(\mathbf{f}_2,\mathbf{f}_2)=-1$. We only need to prove that $T(\mathbf{f}_1,\mathbf{f}_2)=0$. For any $\alpha\in \zR$, taking $ \beta =\overline{T(\mathbf{f}_1,\mathbf{f}_2)}$, we have
\begin{eqnarray*}
\alpha^2 + 2\alpha |T(\mathbf{f}_1,\mathbf{f}_2)|^2+|T(\mathbf{f}_1,\mathbf{f}_2))|^4&=&\alpha^2 + 2\alpha\beta T(\mathbf{f}_1,\mathbf{f}_2)+(\beta T(\mathbf{f}_1,\mathbf{f}_2))^2\\
&=&(\alpha + \beta T(\mathbf{f}_1,\mathbf{f}_2))^2\\
&=&(T(\mathbf{f}_1,\alpha \mathbf{f}_1+\beta \mathbf{f}_2))^2\\
&\leq& \Vert \alpha \mathbf{f}_1+\beta \mathbf{f}_2 \Vert^2\\
&=& \alpha^2+|\beta|^2.
\end{eqnarray*}
Therefore $2\alpha |T(\mathbf{f}_1,\mathbf{f}_2)|^2+|T(\mathbf{f}_1,\mathbf{f}_2))|^4\leq |\beta|^2$ for all $\alpha \in \zR$, which implies that $T(\mathbf{f}_1,\mathbf{f}_2)=0$.

The real case can also be proved using the eigendecomposition for real symmetric matrices and that $T$ cannot be the inner product.
Now let us prove that in the real case $F$ is unique. Suppose that $H=\{\mathbf{h}_1,\mathbf{h}_2\}$ is an orthonormal basis such that
$$[T]_H=
\left(\begin{array}{lr}
1 	& 	 0	 \\
0	&	-1		\\
\end{array}\right).$$
If $[\mathbf{h}_1]_F=(a,b)$, then
\begin{eqnarray*}
1&=&T(\mathbf{h}_1,\mathbf{h}_1)\\
&=& [\mathbf{h}_1]_F
[T]_F
[\mathbf{h}_1]_F^t \\
&=& a^2-b^2\
\end{eqnarray*}
Therefore, $a=\pm 1$ and $b=0$. This means that  $h_1=\pm f_1$ and, being both bases orthonormal, we also have $h_2=\pm f_2$.
\end{proof}

\begin{rem} Notice that uniqueness fails in the complex case. If $F=\{\mathbf{f}_1,\mathbf{f}_2\}$ is an orthonormal basis as in the Lemma \ref{lema bilineales}, then $H=\left\{\frac{1}{\sqrt{2}}\mathbf{f}_1+\frac{i}{\sqrt{2}}\mathbf{f}_2, -\frac{i}{\sqrt{2}}\mathbf{f}_1 +\frac{1}{\sqrt{2}}\mathbf{f}_2\right\}$ is another orthonormal basis such that
$$[T]_H=\left(\begin{array}{lr}
1 	& 	 0	 \\
0	&	-1		\\
\end{array}\right).$$
\end{rem}

\medskip
Now we show that for any two vectors, there exists some bilinear form attaining its norm at them.

\begin{prop}\label{bilineales} Given two norm one vectors $\mathbf{x}, \mathbf{y}$ on a Hilbert space $\mathcal H$, there is a symmetric bilinear form $T:\mathcal H \times \HH\rightarrow \zK$ that attains its norm at $(\mathbf{x},\mathbf{y})$.
\end{prop}
\begin{proof}
If $\mathbf{x}\parallel \mathbf{y}$ it is enough to consider $T$ defined as
$$T(\mathbf{w}_1,\mathbf{w}_2):=\langle \mathbf{w}_1,\mathbf{x}\rangle\,\langle \mathbf{w}_2,\mathbf{x}\rangle.$$
For $\mathbf{x}\nparallel \mathbf{y}$, take $\mathbf{g}_1, \mathbf{g}_2$ as in the proof of Lemma \ref{lema bilineales} and $T$ defined as
\begin{eqnarray*}
T(\mathbf{w}_1,\mathbf{w}_2)&:=& [\operatorname{Proy}(\mathbf{w}_1)]_G\left(\begin{array}{lr}
1 	& 	 0	 \\
0	&	-1		\\
\end{array}\right)[\operatorname{Proy}(\mathbf{w}_2)]_G^t\\
&=&\langle \mathbf{w}_1, \mathbf{g}_1\rangle \langle\mathbf{w}_2, \mathbf{g}_1\rangle -\langle\mathbf{w}_1, \mathbf{g}_2\rangle\langle\mathbf{w}_2, \mathbf{g}_2\rangle, \
\end{eqnarray*}
where $\operatorname{Proy}:\mathcal H\rightarrow \operatorname{span}\{\mathbf{x},\mathbf{y}\}$ is the orthogonal projection.\end{proof}



\section{Multilinear forms: The Complex Case}\label{seccion compleja}

In this section we show the complex case of Theorem~\ref{teorema_principal}. Note that in this case, one of the implications of the theorem is trivial.

\begin{prop}\label{prop caso complejo} Let $\mathcal H$ be a complex Hilbert space  and $k\ge 3$. If a nonzero $T\in  \mathcal L _s(^k\mathcal H)$ attains its norm at $(\mathbf{x}_1,\ldots,\mathbf{x}_k)$, then
$$\dim (\operatorname{span}\{\mathbf{x}_1,\ldots,\mathbf{x}_k\})=1.$$
\end{prop}

\begin{proof}
We will prove the case $k=3$, the general case follows by induction on $k$, fixing one variable of $T$. Suppose $T$ has norm one and that attains its norm at $(\mathbf{x}_1,\mathbf{x}_2,\mathbf{x}_3)$, with $\mathbf{x}_2\nparallel \mathbf{x}_3$

We may assume $\mathbf{x}_1=\mathbf x_2$. Indeed, if $\mathbf{x}_1\parallel \mathbf{x}_2$ just take a multiple of one of them. If $\mathbf{x}_1\not\parallel \mathbf{x}_2$, applying Lemma \ref{lema bilineales} to the bilinear form $T(\,\,\cdot\,\,, \,\,\cdot\,\,, \mathbf{x}_3)$, we may replace $(\mathbf{x}_1,\mathbf{x}_2)$ by $\left(\mathbf{g}_i,\mathbf{g}_i\right)$, with $i$ such that $\mathbf{g}_i\nparallel \mathbf{x}_3$. Also, working with the restriction of $T$ to $\operatorname{span}\{\mathbf{x}_2,\mathbf{x}_2,\mathbf{x}_3\}$, there is no harm in assuming  $\mathcal H =\operatorname{span}\{\mathbf{x}_2,\mathbf{x}_2,\mathbf{x}_3\}$.

By Lemma \ref{lema bilineales} applied to $T(\mathbf{x}_2,\,\,\cdot\,\,, \,\,\cdot\,\,)$, there is an orthonormal basis $F=\{\mathbf{f}_1,\mathbf{f}_2\}$ of $\mathcal{H}$ such that
$$
\begin{array}{c c c}
T(\mathbf{x}_2,\mathbf{f}_1,\mathbf{f}_1)& =&\,\,\,\,1\\
T(\mathbf{x}_2,\mathbf{f}_2,\mathbf{f}_2)&=&-1\\
T(\mathbf{x}_2,\mathbf{f}_1,\mathbf{f}_2)&=& \,\,\,\,\,\,0.\
\end{array}
$$

Consider now the bilinear forms $T(\,\,\cdot\,\,,\,\,\cdot\,\,,\mathbf{f}_1)$ and $T(\,\,\cdot\,\,,\,\,\cdot\,\,,\mathbf{f}_2)$. We have that $T(\,\,\cdot\,\,,\,\,\cdot\,\,,\mathbf{f}_1)$ attains its norm at $(\mathbf{x}_2,\mathbf{f}_1)$. Since $\mathbf{f}_1\nparallel \mathbf{x}_2$ (see proof of Lemma \ref{lema bilineales}), the matrix $A_1:= [T(\,\,\cdot\,\,, \,\,\cdot\,\,, \mathbf{f}_1)]_F$ is unitary. Similarly, $A_2:= [T(\,\,\cdot\,\,, \,\,\cdot\,\,, \mathbf{f}_2)]_F$ is also unitary.

Take $\alpha,\beta \in \zC$ such that $|\alpha|^2+|\beta|^2=1$. Consider the norm one vectors $\mathbf{v}:=\alpha \mathbf{f}_1+\beta \mathbf{f}_2$ and $\mathbf{w}:=\overline{\alpha}\mathbf{f}_1-\overline{\beta}\mathbf{f}_2$. The bilinear form $T(\,\,\cdot\,\,,\,\,\cdot\,\,,\mathbf{v})$ attains its norm at $(\mathbf{x}_2,\mathbf{w})$:
\begin{eqnarray*}
T(\mathbf{x}_2,\overline{\alpha}\mathbf{f}_1-\overline{\beta}\mathbf{f}_2, \alpha \mathbf{f}_1+\beta \mathbf{f}_2)&=& |\alpha|^2T(\mathbf{x}_2,\mathbf{f}_1,\mathbf{f}_1)-|\beta|^2T(\mathbf{x}_2,\mathbf{f}_2,\mathbf{f}_2)\\
 & &+ (\overline{\alpha}\beta-\alpha \overline{\beta}) T(\mathbf{x}_2,\mathbf{f}_1,\mathbf{f}_2)\\
&=& 1.\
\end{eqnarray*}
Then, if $\mathbf{x}_2 \nparallel \mathbf{w}$, we have that the matrix
\begin{equation}\nonumber
[T(\,\,\cdot\,\,, \,\,\cdot\,\,,  \mathbf{v})]_F=\alpha A_1+\beta A_2
\end{equation}
is unitary. Therefore
\begin{eqnarray*}
\operatorname{Id}&=&(\alpha A_1+\beta A_2)(\alpha A_1+\beta A_2)^*\\
&=& |\alpha|^2 \operatorname{Id} + |\beta|^2 \operatorname{Id} + \alpha \overline{\beta}A_1A_2^*+\overline{\alpha}\beta A_1^*A_2\\
&=& \operatorname{Id} + \alpha \overline{\beta}A_1A_2^*+\overline{\alpha}\beta A_1^*A_2.\
\end{eqnarray*}
This holds for any $\alpha$ and $\beta$ for which $\mathbf{x}_2 \nparallel \mathbf{w}$. From this it is easy to conclude that $$A_1A_2^*=A_1^*A_2=0,$$
which cannot happen, since $A_1$ and $A_2$ are unitary. The contradiction comes from the assumption $\mathbf{x}_2 \nparallel \mathbf{x}_3$.
\end{proof}

This Proposition gives us the complex case of Theorem \ref{teorema_principal}.



\section{Multilinear forms: The Real Case}\label{seccion real}

\begin{prop}\label{prop ida caso real} Let $\mathcal H$ be a real Hilbert space. If a nonzero $T\in  \mathcal L _s(^k\mathcal H)$  attains its norm at $(\mathbf{x}_1,\ldots,\mathbf{x}_k)$, then
$$\dim (\operatorname{span}\{\mathbf{x}_1,\ldots,\mathbf{x}_k\})\leq 2.$$
\end{prop}
\begin{proof}
As in the proof of Proposition \ref{prop caso complejo}, let us assume that $k=3$ and $\Vert T\Vert=1$. Let us suppose that $T$ attains its norm at $(\mathbf{x}_1,\mathbf{x}_2,\mathbf{x}_3)$ with
$\dim (\operatorname{span}\{\mathbf{x}_1,\mathbf{x}_2,\mathbf{x}_3\})=3$
and arrive to a contradiction. To simplify the notation we also assume that $\HH=\operatorname{span}\{\mathbf{x}_1,\mathbf{x}_2,\mathbf{x}_3\}=\zR^3$.

By Lemma \ref{lema bilineales} applied to $T(\,\,\cdot\,\,, \,\,\cdot\,\,, \mathbf{x}_3)$, there is an orthonormal basis $\{\mathbf{e}_1,\mathbf{e}_2\}$ of $\operatorname{span}\{\mathbf{x}_1,\mathbf{x}_2\}$ such that
$$\begin{array}{c c c}
T(\mathbf{e}_1,\mathbf{e}_1,\mathbf{x}_3)& =&\,\,\,\,1\\
T(\mathbf{e}_2,\mathbf{e}_2,\mathbf{x}_3)&=&-1\\
T(\mathbf{e}_1,\mathbf{e}_2,\mathbf{x}_3)&=&\,\,\,\,\,\, 0.\
\end{array}$$
We extend $\{\mathbf{e}_1,\mathbf{e}_2\}$ to an orthonormal basis $E=\{\mathbf{e}_1,\mathbf{e}_2,\mathbf{e}_3\}$ of $\mathcal H$ and, for simplicity, suppose $E$ is the canonical basis. Given that $\mathbf{x}_3\notin \operatorname{span}\{\mathbf{x}_1,\mathbf{x}_2\}=\operatorname{span}\{\mathbf{e}_1,\mathbf{e}_2\}$ we can write $\mathbf{x}_3=(\alpha, \beta,\gamma)$ with $\gamma \neq 0.$

Since the norm of the linear function $T(\,\,\cdot\,\,, \mathbf{e}_1,\mathbf{x}_3)$ is at most one and
$$T(\,\,\cdot\,\,, \mathbf{e}_1,\mathbf{x}_3)(\mathbf{e}_1)=T(\mathbf{e}_1, \mathbf{e}_1,\mathbf{x}_3)=1,$$
we must have
\begin{equation}\label{ecuacion e1 x3}
T(\,\,\cdot\,\,, \mathbf{e}_1,\mathbf{x}_3)=\langle \,\,\cdot\,\, , \mathbf{e}_1\rangle.
\end{equation}
Using similar arguments we have that
\begin{eqnarray}
T(\,\,\cdot\,\,, \mathbf{e}_1,\mathbf{e}_1)&=&\langle \,\,\cdot\,\, , \mathbf{x}_3\rangle \label{ecuacion e1 e1} \\
T(\,\,\cdot\,\,, \mathbf{e}_2,\mathbf{x}_3)&=&\langle \,\,\cdot\,\, , -\mathbf{e}_2\rangle \label{ecuacion e2 x3}\\
T(\,\,\cdot\,\,, \mathbf{e}_2,\mathbf{e}_2)&=&\langle \,\,\cdot\,\, , -\mathbf{x}_3\rangle.\label{ecuacion e2 e2}
\end{eqnarray}
Finally, using \eqref{ecuacion e1 x3} and \eqref{ecuacion e2 x3}, we have
$$T(\,\,\cdot\,\,, \mathbf{e}_3,\mathbf{x}_3)(\mathbf{e}_1)=T(\mathbf{e}_1, \mathbf{e}_3,\mathbf{x}_3)=0,$$
$$T(\,\,\cdot\,\,, \mathbf{e}_3,\mathbf{x}_3)(\mathbf{e}_2)=T(\mathbf{e}_2, \mathbf{e}_3,\mathbf{x}_3)=0;$$
which means that
\begin{equation}\label{ecuacion e3 x3}
T(\,\,\cdot\,\,, \mathbf{e}_3,\mathbf{x}_3)=\langle \,\,\cdot\,\, , t\mathbf{e}_3\rangle.
\end{equation}
for some $t\in [-1,1]$.
Then, using \eqref{ecuacion e1 x3}, \eqref{ecuacion e2 x3} and \eqref{ecuacion e3 x3}, we have
\begin{equation}\label{ecuacion x3 x3}
T(\,\,\cdot\,\,, \mathbf{x}_3,\mathbf{x}_3)=T(\,\,\cdot\,\,, \mathbf{x}_3,\alpha \mathbf{e}_1+\beta\mathbf{e}_2+\gamma \mathbf{e}_3)=\langle \,\,\cdot\,\, , (\alpha, -\beta, t\gamma)\rangle.
\end{equation}
With all these equations at hand we can finally arrive to a contradiction. To do this, we consider first the linear function $T(\,\,\cdot\,\,, \mathbf{x}_3+ \mathbf{e}_1,\mathbf{x}_3+\mathbf{e}_1)$ which has norm at most $\Vert \mathbf{x}_3+ \mathbf{e}_1\Vert^2$. By \eqref{ecuacion e1 e1}, \eqref{ecuacion e2 x3} and \eqref{ecuacion x3 x3} we have
\begin{eqnarray}
T(\,\,\cdot\,\,, \mathbf{x}_3+ \mathbf{e}_1,\mathbf{x}_3+\mathbf{e}_1) &=& 2T(\,\,\cdot\,\,, \mathbf{x}_3,\mathbf{e}_1)+T(\,\,\cdot\,\,, \mathbf{x}_3,\mathbf{x}_3)+T(\,\,\cdot\,\,,  \mathbf{e}_1,\mathbf{e}_1) \nonumber \\
&=& \langle \,\,\cdot\,\, , 2\mathbf{e}_1+(\alpha, -\beta, t\gamma) +\mathbf{x}_3\rangle \nonumber \\
&=& \langle \,\,\cdot\,\, , (2+2\alpha, 0, \gamma(1+t)) \rangle. \nonumber\label{ecuacion x3+e1}\
\end{eqnarray}
Given that $\alpha^2+\beta^2+\gamma^2=\Vert \mathbf{x}_3\Vert=1$, then
\begin{eqnarray*}
|2+2\alpha|&=&\Vert \mathbf{x}_3+ \mathbf{e}_1\Vert^2  \\
&\geq & \Vert T(\,\,\cdot\,\,, \mathbf{x}_3+ \mathbf{e}_1,\mathbf{x}_3+\mathbf{e}_1)\Vert  \\
&=& \Vert \langle \,\,\cdot\,\,, (2+2\alpha,0, \gamma(1+t))\rangle \Vert \\
&=& \Vert (2+2\alpha,0, \gamma(1+t))\Vert. \\
\end{eqnarray*}
Therefore, since $\gamma \neq 0$, $t$ must be $-1$. But we can also consider the linear function
\begin{eqnarray}
\tilde T(\,\,\cdot\,\,, \mathbf{x}_3+ \mathbf{e}_2,\mathbf{x}_3+\mathbf{e}_2) &=& 2\tilde T(\,\,\cdot\,\,, \mathbf{x}_3,\mathbf{e}_2)+\tilde T(\,\,\cdot\,\,, \mathbf{x}_3,\mathbf{x}_3)+\tilde T(\,\,\cdot\,\,,  \mathbf{e}_2,\mathbf{e}_2) \nonumber \\
&=& \langle \,\,\cdot\,\, , -2\mathbf{e}_2+(\alpha, -\beta, t\gamma) -\mathbf{x}_3\rangle \nonumber \\
&=& \langle \,\,\cdot\,\, , (0, -2-2\beta, \gamma(t-1)) \rangle \nonumber\label{ecuacion x3+e2}\
\end{eqnarray}
and proceed similarly:
\begin{eqnarray*}
|2+2\beta|&=&\Vert \mathbf{x}_3+ \mathbf{e}_2\Vert^2  \\
&\geq & \Vert \tilde T(\,\,\cdot\,\,, \mathbf{x}_3+ \mathbf{e}_2,\mathbf{x}_3+\mathbf{e}_2)\Vert  \\
&=& \Vert (0, -2-\beta, \gamma(t-1))\Vert. \\
\end{eqnarray*}
This means thas $t$ must be $1$, a contradiction which comes from the assumption that $\gamma$ is nonzero.
\end{proof}

\medskip
Now we turn our attention to the other implication in Theorem~\ref{teorema_principal}: the existence of the multilinear form attaining its norm at the desired set. This implication, which is trivial in the complex case (the vectors are collinear), is rather complicated in the real case, as we will see.

\begin{lem}\label{lema polinomios norma 1} Given $n\in \zK$, the polynomials $P,Q:\zR^2 \rightarrow \zR$ defined as
$$P(x_1,x_2):= \sum_{l=0}^{\left[\frac{k}{2}\right]} \binom{k}{2l}(-1)^l x_1 ^{k-2l}x_2^{2l} \mbox{ and } Q(x_1,x_2):= \sum_{l=0}^{\left[\frac{k-1}{2}\right]} \binom{k}{2l+1}(-1)^l x_1 ^{k-(2l+1)}x_2^{2l+1} $$
have norm one.
\end{lem}
\begin{proof} The argument we use to prove this Lemma is  essentially the same used to see that the Chebyshev polynomials have norm one (see for example \cite{rivlin1974chebyshev} Chapter I).
The norm of $P$ can be computed as follows
$$
\Vert P\Vert =\sup\left\{\left\vert \sum_{l=0}^{\left[\frac{k}{2}\right]} \binom{k}{2l}(-1)^l \cos ^{k-2l} (\theta)\sin^{2l}(\theta) \right\vert : 0\leq \theta \leq 2\pi\right\}.
$$
Given that for any $\theta$
$$1\geq \vert \operatorname{Re}(e^{ik\theta})\vert =\left\vert \sum_{l=0}^{\left[\frac{k}{2}\right]} \binom{k}{2l}(-1)^l \cos ^{k-2l} (\theta)\sin^{2l}(\theta) \right\vert,$$
we conclude that $\Vert P\Vert \leq 1$.  For the other inequality we have
$$\Vert P\Vert \geq \vert P(1,0)\vert =1.
$$

A similar analysis, using $\operatorname{Im}(e^{ik\theta})$ instead of $\operatorname{Re}(e^{ik\theta})$, shows that $Q$ has norm one.
\end{proof}

\begin{prop}\label{prop vuelta caso real} Given $\mathbf{x}_1,\ldots,\mathbf{x}_k$  norm one vectors on a real Hilbert space $\mathcal H$ such that $\dim (\operatorname{span}\{\mathbf{x}_1,\ldots,\mathbf{x}_k\})\leq 2,$ there is a nonzero $T\in  \mathcal L _s(^k\mathcal H)$ attaining its norm at $(\mathbf{x}_1,\ldots, \mathbf{x}_k)$.
\end{prop}
\begin{proof} By Proposition \ref{bilineales}, we only need to prove the case $k>2$.
To simplify notation assume that  $\mathcal H =\zR^2$. We will prove this result in three steps. In what follows, $\{\mathbf{e}_1, \mathbf{e}_2\}$ stands for the canonical basis of $\zR^2$.

\paragraph{\textbf{Step I:  $\mathbf{x}_1,\ldots,\mathbf{x}_k\in \{\mathbf{e}_1, \mathbf{e}_2,-\mathbf{e}_1, -\mathbf{e}_2\}$.}} Given that if $T$ attains its norm at $(\mathbf{x}_1,\ldots, \mathbf{x}_k)$ then $T$ also attains its norm at $(\pm\mathbf{x}_1,\ldots, \pm\mathbf{x}_k)$ it is enough to consider the case  $\mathbf{x}_1,\ldots,\mathbf{x}_k\in \{\mathbf{e}_1, \mathbf{e}_2\}$. Suppose that we have $i$ times the vector $\mathbf{e}_1$ and $j$ times the vector $\mathbf{e}_2$. We write $(\mathbf{x}_1,\ldots, \mathbf{x}_k)=(\mathbf{e}_1^i, \mathbf{e}_2^j)$, notation that  will stay with us for the rest of the proof. Suppose that either $i$ or $j$ are even numbers.  To find $T\in  \mathcal  L _s(^k\mathcal \zR^2)$ attaining its norm at $(\mathbf{e}_1^i, \mathbf{e}_2^j)$ we will work with the space of $k$-homogeneous polynomials $ \mathcal P (^k\mathcal \zR^2)$ which, as mentioned before, is isometric to $\mathcal L _s(^k\mathcal \zR^2)$ (see \eqref{eq-isometria}).

Consider  the norm one $k$-homogeneous polynomial $P:\zR^2\rightarrow \zR$ from Lemma \ref{lema polinomios norma 1}:
$$P(x_1,x_2)= \sum_{l=0}^{\left[\frac{k}{2}\right]} \binom{k}{2l}(-1)^l x_1 ^{k-2l}x_2^{2l}.$$

Using Theorem 2.2 from \cite{minc1984permanents} it is easy to see that the  norm one $k$-linear symmetric form $\widecheck{P}$ associated to $P$ attains its norm at $(\mathbf{e}_1^i, \mathbf{e}_2^j)$:
\begin{eqnarray*}
 \widecheck{P}(\mathbf{e}_1^i, \mathbf{e}_2^j) &=&  \sum_{l=0}^{\frac{k}{2}} \binom{k}{2l}(-1)^l \widecheck{(x_1 ^{k-2l}x_2^{2l})}(\mathbf{e}_1^i, \mathbf{e}_2^j)  \\
 &=&\binom{k}{j}(-1)^j \widecheck{(x_1 ^{k-2l}x_2^{2l})}(\mathbf{e}_1^i, \mathbf{e}_2^j)  \\
&=&  1.\
\end{eqnarray*}

If both $i$ and $j$ are odd number we can use polynomial $Q$ from Lemma \ref{lema polinomios norma 1}, instead of $P$.

\paragraph{\textbf{Step II:} For the second step we will consider the following increasing sequence of sets
$$D_n= \left\{(\cos \,\theta,\sin \,\theta) : 	{\theta = \frac{\pi\,l}{2^{n}}, l=1,\ldots,2^{n+1}}
\right\}.$$
In this step we prove the result for $n\in \zN$ and $\mathbf{x}_1,\ldots,\mathbf{x}_k\in D_n$.} We do this by induction on $n$. Notice that the case $n=1$ is Step I. Suppose the result holds for $n$. Given
norm one vectors $\mathbf{x}_1,\ldots,\mathbf{x}_k\in D_{n+1}$ we need to find $T\in  \mathcal L _s(^k\mathcal \zR^2)$ attaining its norm at $(\mathbf{x}_1,\ldots, \mathbf{x}_k)$. Since $\mathbf{x}_1,\ldots,\mathbf{x}_k\in D_{n+1}$, we have that for each $i$ there are $\mathbf{y}_i,\mathbf{z}_i\in D_n$ such that $\mathbf{x}_i= \frac{\mathbf{y}_i+\mathbf{z}_i}{\Vert \mathbf{y}_i+\mathbf{z}_i\Vert}$. By inductive hypothesis there is $L\in\mathcal L _s(^{2k}\mathcal \zR^2)$ attaining its norm at $(\mathbf{y}_1,\mathbf{z}_1,\ldots, \mathbf{y}_k,\mathbf{z}_k)$. Fixing all but two variable at a time it is easy to see that $L$ also attains its norm at
$$\left( \left(\frac{\mathbf{y}_1+\mathbf{z}_1}{\Vert \mathbf{y}_1+\mathbf{z}_1\Vert}\right)^2,\ldots,  \left(\frac{\mathbf{y}_k+\mathbf{z}_k}{\Vert \mathbf{y}_k+\mathbf{z}_k\Vert}\right)^2\right)=(\mathbf{x}_1^2,\ldots, \mathbf{x}_k^2).$$
Then, the $k$-linear form $L(\,\,\cdot\,\,, \ldots, \,\,\cdot\,\,, \mathbf{x}_1,\ldots,\mathbf{x}_k)$ attains its norm at $(\mathbf{x}_1,\ldots,\mathbf{x}_k)$.

\paragraph{\textbf{Step III: general case.}} Given norm one vectors $\mathbf{x}_1,\ldots,\mathbf{x}_k$ and $m\in \zN$, take  $\mathbf{x}_1^{(m)},\ldots,\mathbf{x}_k^{(m)}\in D_n$ (with $n=n(m)$) and $T_m \in  \mathcal L _s(^k\mathcal \zR^2)$ of norm one such that
\begin{itemize}
\item $\Vert \mathbf{x}_i^{(m)} - \mathbf{x}_i\Vert <\frac{1}{m}$ for $i=1,\ldots,k$.
\item $T_n(\mathbf{x}_1^{(m)},\ldots,\mathbf{x}_k^{(m)})=1$.
\end{itemize}

Take  $\{T_{m_j}\}_{j\in \zN}$  a convergent subsequence of  $\{T_{m}\}_{m\in \zN}$ and let $T\in  \mathcal L _s(^k\mathcal \zR^2)$ be its limit (note that $T$ has norm one).
It is easy to check that
$$ 1= T_{m_j}(\mathbf{x}_1^{(m_j)},\ldots,\mathbf{x}_k^{(m_j)}) \to T(\mathbf{x}_1,\ldots,\mathbf{x}_k).$$ Therefore,  $T$ attains its norm at  $(\mathbf{x}_1,\ldots,\mathbf{x}_k)$.
\end{proof}

Combining Propositions \ref{prop ida caso real} and \ref{prop vuelta caso real} we obtain the real case of Theorem \ref{teorema_principal}.



%

\end{document}